\documentclass[english]{article}
\usepackage[T1]{fontenc}
\usepackage[latin9]{inputenc}
\usepackage{amsmath}
\usepackage{amsthm}

\makeatletter
\theoremstyle{plain}
\newtheorem{thm}{\protect\theoremname}
\theoremstyle{plain}
\newtheorem{lem}[thm]{\protect\lemmaname}
\ifx\proof\undefined
\newenvironment{proof}[1][\protect\proofname]{\par
	\normalfont\topsep6\p@\@plus6\p@\relax
	\trivlist
	\itemindent\parindent
	\item[\hskip\labelsep\scshape #1]\ignorespaces
}{%
	\endtrivlist\@endpefalse
}
\providecommand{\proofname}{Proof}
\fi
\theoremstyle{plain}
\newtheorem{cor}[thm]{\protect\corollaryname}
\newcommand{\lyxaddress}[1]{
	\par {\raggedright #1
	\vspace{1.4em}
	\noindent\par}
}

\@ifundefined{date}{}{\date{}}
\makeatother

\usepackage{babel}
\providecommand{\corollaryname}{Corollary}
\providecommand{\lemmaname}{Lemma}
\providecommand{\theoremname}{Theorem}

\begin{document}
\title{Biharmonic hypersurfaces in hemispheres}
\author{Matheus Vieira}
\maketitle
\begin{abstract}
In this paper we consider the Balmu\c{s}-Montaldo-Oniciuc's conjecture
in the case of hemispheres. We prove that a compact non-minimal biharmonic
hypersurface in a hemisphere of $S^{n+1}$ must be the small hypersphere
$S^{n}\left(1/\sqrt{2}\right)$, provided that $n^{2}-H^{2}$ does
not change sign.
\end{abstract}

\section{Introduction}

It is well known that minimal hypersurfaces can be seen as hypersurfaces
whose canonical inclusion is a harmonic map. Thus it is natural to
study hypersurfaces whose canonical inclusion is a biharmonic map,
known as biharmonic hypersurfaces (for more information see Section
2). From the point of view of finding new examples and classification
results, the theory of biharmonic hypersurfaces seems to be more interesting
when the ambient space has positive curvature. There are many papers
studying biharmonic hypersurfaces in the sphere (for example \cite{BMO2008},
\cite{BO2012}, \cite{CMO2001}, \cite{CMO2002}, \cite{C1993}, \cite{FH2018},
\cite{LM2017}, \cite{M2017}, \cite{On2003}, \cite{On2012}).

Let $S^{n+1}$ be the unit Euclidean sphere. Balmu\c{s}-Montaldo-Oniciuc
\cite{BMO2008} conjectured that a non-minimal biharmonic hypersurface
in $S^{n+1}$ must be an open part of the small hypersphere $S^{n}\left(1/\sqrt{2}\right)$
of radius $1/\sqrt{2}$ or of a generalized Clifford torus $S^{k}\left(1/\sqrt{2}\right)\times S^{n-k}\left(1/\sqrt{2}\right)$
with $k\neq n/2$. Since a generalized Clifford torus cannot lie in
a hemisphere, it is natural to ask whether a compact non-minimal biharmonic
hypersurface in a hemisphere must be the small hypersphere. We give
an affirmative answer to this question when $n^{2}-H^{2}$ does not
change sign.
\begin{thm}
\label{thm:main}Let $M^{n}$ be a compact biharmonic hypersurface
in a closed hemisphere of $S^{n+1}$. If $n^{2}-H^{2}$ does not change
sign then either $M^{n}$ is the equator of the hemisphere or it is
the small hypersphere $S^{n}\left(1/\sqrt{2}\right)$.
\end{thm}
Note that $n^{2}-H^{2}\geq0$ when $\left|A\right|^{2}\leq n$ (by
the Cauchy-Schwarz inequality) or $H$ is constant (by Theorem \ref{thm:biharmonicequations}).
This is related to some previous results (see for example Section
1.4 in \cite{On2012}).

The case $n^{2}-H^{2}\leq0$ was proved in a direct way by Balmu\c{s}-Oniciuc
(Corollary 3.3 in \cite{BO2012}). Our proof is completely different.
It is based on a formula for the bilaplacian of the restriction of
a function defined on the ambient space (Theorem \ref{thm:bilaplacianrestriction}).

The author would like to thank Detang Zhou, Cezar Oniciuc and Dorel
Fetcu for their support.

\section{Preliminaries}

First we give the notations and conventions of the paper.

Let $M$ be a Riemannian manifold and let $\nabla$ be the Levi-Civita
connection. We define the Riemann curvature by
\[
Riem\left(u,v\right)w=\nabla_{u}\nabla_{v}w-\nabla_{v}\nabla_{u}w-\nabla_{\left[u,v\right]}w.
\]
If $f$ is a function on $M$ we define the Hessian of $f$ by
\[
\nabla\nabla f\left(u,v\right)=\left\langle \nabla_{u}\nabla f,v\right\rangle ,
\]
and the Laplacian of $f$ by
\[
\Delta f=tr_{M}\nabla\nabla f.
\]

Let $\bar{M}^{n+1}$ be a Riemannian manifold and let $M^{n}$ be
a hypersurface in $\bar{M}^{n+1}$. We define the second fundamental
form by
\[
A\left(u,v\right)=\left\langle \bar{\nabla}_{u}v,N\right\rangle ,
\]
and the mean curvature by
\[
H=tr_{M}A,
\]
where $N$ is the normal vector. Note that $H$ is not normalized.
We denote the geometric quantities of $\bar{M}^{n+1}$ with a bar.
For example $\bar{\nabla}$ is the Levi-Civita connection of $\bar{M}^{n+1}$,
$\bar{R}iem$ is the Riemann curvature of $\bar{M}^{n+1}$ and $\bar{R}ic$
is the Ricci curvature of $\bar{M}^{n+1}$. We denote the geometric
quantities of $M^{n}$ without a bar. For example $\nabla$ is the
Levi-Civita connection of $M^{n}$ and $\Delta$ is the Laplacian
of $M^{n}$.

Now we recall the definition of biharmonic maps and hypersurfaces.

Let $\phi:M\to N$ be a map between Riemannian manifolds. We say that
$\phi$ is harmonic if it is a critical point of the functional 
\[
E\left(\phi\right)=\int_{M}\left|d\phi\right|^{2}.
\]
Critical points of $E$ satisfy $\tau\left(\phi\right)=0$, where
\[
\tau\left(\phi\right)=tr\nabla d\phi.
\]
We say that $\phi$ is biharmonic if it is a critical point of the
functional
\[
E_{2}\left(\phi\right)=\int_{M}\left|\tau\left(\phi\right)\right|^{2}.
\]
Critical points of $E_{2}$ satisfy $\tau_{2}\left(\phi\right)=0$,
where
\[
\tau_{2}\left(\phi\right)=\Delta\tau\left(\phi\right)+tr\bar{R}iem\left(\tau\left(\phi\right),d\phi\right)d\phi.
\]
Now let $\bar{M}^{n+1}$ be a Riemannian manifold and let $M^{n}$
be a hypersurface in $\bar{M}^{n+1}$. It is well known that $M^{n}$
is minimal if and only if the canonical inclusion is a harmonic map.
We say that $M^{n}$ is biharmonic if the canonical inclusion is a
biharmonic map. For more information about harmonic and biharmonic
maps and submanifolds see \cite{ES1964}, \cite{J1986A}, \cite{J1986B}.

In the rest of this section we obtain a formula for the bilaplacian
of the restriction of a function defined on the ambient space.

The next result was applied to submanifolds of the sphere for the
first time by Oniciuc (Theorem 3.1 in \cite{On2003}). Later it was
applied to hypersurfaces in general ambient spaces by Ou (Theorem
2.1 in \cite{Ou2010}). See also Remark 4.10 in \cite{LMO2008}. This
is an important result in the theory of biharmonic submanifolds.
\begin{thm}
\label{thm:biharmonicequations}(\cite{On2003}, \cite{Ou2010}) Let
$\bar{M}^{n+1}$ be a Riemannian manifold and let $M^{n}$ be a hypersurface
in $\bar{M}^{n+1}$. Then $M^{n}$ is biharmonic if and only if $B^{N}$
and $B^{T}$ vanish, where
\[
B^{N}=\Delta H-H\left|A\right|^{2}+H\bar{R}ic\left(N,N\right),
\]
and
\[
B^{T}=2A\left(\nabla H\right)+H\nabla H-2H\left(\bar{R}ic\left(N\right)\right)^{T}.
\]
\end{thm}
Here we identify $\left(1,1\right)$ tensors and $\left(0,2\right)$
tensors, and $\left(\cdot\right)^{T}$ is the projection on $TM^{n}$.
In particular
\[
A\left(\nabla H\right)=\sum_{i=1}^{n}A\left(\nabla H,e_{i}\right)e_{i},
\]
and
\[
\left(\bar{R}ic\left(N\right)\right)^{T}=\sum_{i=1}^{n}\bar{R}ic\left(N,e_{i}\right)e_{i},
\]
where $\left\{ e_{i}\right\} _{i=1}^{n}$ is a local orthonormal frame
on $M^{n}$.

The next result is well known. We prove it for the sake of completeness.
\begin{lem}
\label{lem:hessianrestriction}Let $\bar{M}^{n+1}$ be a Riemannian
manifold and let $M^{n}$ be a hypersurface in $\bar{M}^{n+1}$. If
$\bar{f}$ is a function on $\bar{M}^{n+1}$ and $f=\bar{f}|M^{n}$
then
\[
\nabla\nabla f\left(u,v\right)=\bar{\nabla}\bar{\nabla}\bar{f}\left(u,v\right)+\left\langle \bar{\nabla}\bar{f},N\right\rangle A\left(u,v\right).
\]
In particular
\[
\Delta f=tr_{M}\bar{\nabla}\bar{\nabla}\bar{f}+\left\langle \bar{\nabla}\bar{f},N\right\rangle H.
\]
\end{lem}
\begin{proof}
Let $\left\{ e_{i}\right\} $ be a local orthonormal frame on $M$
such that $\nabla_{e_{i}}e_{j}=0$ at a fixed point of $M$. At this
point we have
\begin{align*}
\nabla\nabla f\left(e_{i},e_{j}\right) & =e_{i}e_{j}f\\
 & =e_{i}e_{j}\bar{f},
\end{align*}
and
\begin{align*}
\bar{\nabla}\bar{\nabla}\bar{f}\left(e_{i},e_{j}\right) & =e_{i}e_{j}\bar{f}-\left\langle \bar{\nabla}\bar{f},\left(\bar{\nabla}_{e_{i}}e_{j}\right)^{\perp}\right\rangle \\
 & =e_{i}e_{j}\bar{f}-\left\langle \bar{\nabla}\bar{f},N\right\rangle A\left(e_{i},e_{j}\right).
\end{align*}
The result follows from these equations.
\end{proof}
The proof of Theorem \ref{thm:main} is based on next result, which
may be of independent interest.
\begin{thm}
\label{thm:bilaplacianrestriction}Let $\bar{M}^{n+1}$ be a Riemannian
manifold and let $M^{n}$ be a hypersurface in $\bar{M}^{n+1}$. If
$\bar{f}$ is a function on $\bar{M}^{n+1}$ and $f=\bar{f}|M^{n}$
then
\begin{align*}
\Delta\Delta f & =\Delta\left(tr_{M}\bar{\nabla}\bar{\nabla}\bar{f}\right)+Htr_{M}\left(\bar{\nabla}_{N}\bar{\nabla}\bar{\nabla}\bar{f}\right)+H^{2}\bar{\nabla}\bar{\nabla}\bar{f}\left(N,N\right)-2H\left\langle \bar{\nabla}\bar{\nabla}\bar{f},A\right\rangle \\
 & +2\bar{\nabla}\bar{\nabla}\bar{f}\left(\nabla H,N\right)+\left\langle B^{N}N-B^{T},\bar{\nabla}\bar{f}\right\rangle .
\end{align*}
\end{thm}
\begin{proof}
By Lemma \ref{lem:hessianrestriction} we have
\[
\Delta f=tr_{M}\bar{\nabla}\bar{\nabla}\bar{f}+\left\langle \bar{\nabla}\bar{f},N\right\rangle H.
\]
Taking the Laplacian we obtain
\[
\Delta\Delta f=\Delta\left(tr_{M}\bar{\nabla}\bar{\nabla}\bar{f}\right)+H\Delta\left\langle \bar{\nabla}\bar{f},N\right\rangle +2\left\langle \nabla\left\langle \bar{\nabla}\bar{f},N\right\rangle ,\nabla H\right\rangle +\left\langle \bar{\nabla}\bar{f},N\right\rangle \Delta H.
\]
Let $\left\{ e_{i}\right\} $ be a local orthonormal frame on $M$.
We have
\begin{align*}
e_{i}\left\langle \bar{\nabla}\bar{f},N\right\rangle  & =\left\langle \bar{\nabla}_{e_{i}}\bar{\nabla}\bar{f},N\right\rangle +\left\langle \bar{\nabla}\bar{f},\bar{\nabla}_{e_{i}}N\right\rangle \\
 & =\bar{\nabla}\bar{\nabla}\bar{f}\left(e_{i},N\right)-A\left(e_{i},\nabla f\right).
\end{align*}
We find that
\begin{align*}
\Delta\Delta f & =\Delta\left(tr_{M}\bar{\nabla}\bar{\nabla}\bar{f}\right)+H\Delta\left\langle \bar{\nabla}\bar{f},N\right\rangle +2\bar{\nabla}\bar{\nabla}\bar{f}\left(\nabla H,N\right)\\
 & -2A\left(\nabla H,\nabla f\right)+\left\langle \bar{\nabla}\bar{f},N\right\rangle \Delta H.
\end{align*}
We can assume that $\nabla_{e_{i}}e_{j}=0$ at a fixed point of $M$.
At this point we have
\begin{align*}
\Delta\left\langle \bar{\nabla}\bar{f},N\right\rangle  & =\sum_{i}e_{i}e_{i}\left\langle \bar{\nabla}\bar{f},N\right\rangle \\
 & =\sum_{i}e_{i}\left(\bar{\nabla}\bar{\nabla}\bar{f}\left(e_{i},N\right)\right)-\sum_{i}e_{i}\left(A\left(e_{i},\nabla f\right)\right)\\
 & =\sum_{i}\left(\bar{\nabla}_{e_{i}}\bar{\nabla}\bar{\nabla}\bar{f}\right)\left(e_{i},N\right)+\sum_{i}\bar{\nabla}\bar{\nabla}\bar{f}\left(\left(\bar{\nabla}_{e_{i}}e_{i}\right)^{\perp},N\right)\\
 & +\sum_{i}\bar{\nabla}\bar{\nabla}\bar{f}\left(e_{i},\bar{\nabla}_{e_{i}}N\right)-\sum_{i}\left(\nabla_{e_{i}}A\right)\left(e_{i},\nabla f\right)-\sum_{i}A\left(e_{i},\nabla_{e_{i}}\nabla f\right)\\
 & =\sum_{i}\left(\bar{\nabla}_{e_{i}}\bar{\nabla}\bar{\nabla}\bar{f}\right)\left(e_{i},N\right)+H\bar{\nabla}\bar{\nabla}\bar{f}\left(N,N\right)-\left\langle \bar{\nabla}\bar{\nabla}\bar{f},A\right\rangle \\
 & -\sum_{i}\left(\nabla_{e_{i}}A\right)\left(e_{i},\nabla f\right)-\left\langle \nabla\nabla f,A\right\rangle .
\end{align*}
By Lemma \ref{lem:hessianrestriction} we have
\begin{align*}
\Delta\left\langle \bar{\nabla}\bar{f},N\right\rangle  & =\sum_{i}\left(\bar{\nabla}_{e_{i}}\bar{\nabla}\bar{\nabla}\bar{f}\right)\left(e_{i},N\right)+H\bar{\nabla}\bar{\nabla}\bar{f}\left(N,N\right)-2\left\langle \bar{\nabla}\bar{\nabla}\bar{f},A\right\rangle \\
 & -\sum_{i}\left(\nabla_{e_{i}}A\right)\left(e_{i},\nabla f\right)-\left\langle \bar{\nabla}\bar{f},N\right\rangle \left|A\right|^{2}.
\end{align*}
By the Ricci identity and the Codazzi equation we have
\begin{align*}
\left(\bar{\nabla}_{e_{i}}\bar{\nabla}\bar{\nabla}\bar{f}\right)\left(e_{i},N\right) & =\left(\bar{\nabla}_{e_{i}}\bar{\nabla}\bar{\nabla}\bar{f}\right)\left(N,e_{i}\right)\\
 & =\left(\bar{\nabla}_{N}\bar{\nabla}\bar{\nabla}\bar{f}\right)\left(e_{i},e_{i}\right)-\left\langle \bar{R}iem\left(e_{i},N\right)e_{i},\bar{\nabla}\bar{f}\right\rangle ,
\end{align*}
and
\begin{align*}
\left(\nabla_{e_{i}}A\right)\left(e_{i},\nabla f\right) & =\left(\nabla_{e_{i}}A\right)\left(\nabla f,e_{i}\right)\\
 & =\left(\nabla_{\nabla f}A\right)\left(e_{i},e_{i}\right)+\left\langle \bar{R}iem\left(e_{i},\nabla f\right)e_{i},N\right\rangle .
\end{align*}
We find that
\begin{align*}
\Delta\left\langle \bar{\nabla}\bar{f},N\right\rangle  & =tr_{M}\left(\bar{\nabla}_{N}\bar{\nabla}\bar{\nabla}\bar{f}\right)+\bar{R}ic\left(\bar{\nabla}\bar{f},N\right)+H\bar{\nabla}\bar{\nabla}\bar{f}\left(N,N\right)-2\left\langle \bar{\nabla}\bar{\nabla}\bar{f},A\right\rangle \\
 & -\left\langle \nabla f,\nabla H\right\rangle +\bar{R}ic\left(\nabla f,N\right)-\left\langle \bar{\nabla}\bar{f},N\right\rangle \left|A\right|^{2}.
\end{align*}
We conclude that
\begin{align*}
\Delta\Delta f & =\Delta\left(tr_{M}\bar{\nabla}\bar{\nabla}\bar{f}\right)+Htr_{M}\left(\bar{\nabla}_{N}\bar{\nabla}\bar{\nabla}\bar{f}\right)+H\bar{R}ic\left(\bar{\nabla}\bar{f},N\right)+H^{2}\bar{\nabla}\bar{\nabla}\bar{f}\left(N,N\right)\\
 & -2H\left\langle \bar{\nabla}\bar{\nabla}\bar{f},A\right\rangle -H\left\langle \nabla f,\nabla H\right\rangle +H\bar{R}ic\left(\nabla f,N\right)-H\left\langle \bar{\nabla}\bar{f},N\right\rangle \left|A\right|^{2}\\
 & +2\bar{\nabla}\bar{\nabla}\bar{f}\left(\nabla H,N\right)-2A\left(\nabla H,\nabla f\right)+\left\langle \bar{\nabla}\bar{f},N\right\rangle \Delta H\\
 & =\Delta\left(tr_{M}\bar{\nabla}\bar{\nabla}\bar{f}\right)+Htr_{M}\left(\bar{\nabla}_{N}\bar{\nabla}\bar{\nabla}\bar{f}\right)+H^{2}\bar{\nabla}\bar{\nabla}\bar{f}\left(N,N\right)-2H\left\langle \bar{\nabla}\bar{\nabla}\bar{f},A\right\rangle \\
 & +2\bar{\nabla}\bar{\nabla}\bar{f}\left(\nabla H,N\right)+\left\langle \left(\Delta H-H\left|A\right|^{2}\right)N,\bar{\nabla}\bar{f}\right\rangle -\left\langle 2A\left(\nabla H\right)+H\nabla H,\bar{\nabla}\bar{f}\right\rangle \\
 & +H\bar{R}ic\left(\bar{\nabla}\bar{f},N\right)+H\bar{R}ic\left(\nabla f,N\right).
\end{align*}
The result follows since
\[
\Delta H-H\left|A\right|^{2}=B^{N}-H\bar{R}ic\left(N,N\right),
\]
\[
2A\left(\nabla H\right)+H\nabla H=B^{T}+2H\left(\bar{R}ic\left(N\right)\right)^{T},
\]
and
\[
\bar{R}ic\left(\bar{\nabla}\bar{f},N\right)=\bar{R}ic\left(\nabla f,N\right)+\left\langle \bar{\nabla}\bar{f},N\right\rangle \bar{R}ic\left(N,N\right).
\]
\end{proof}

\section{Proof of the main result}

Proof of Theorem \ref{thm:main}.
\begin{proof}
Since $M$ is biharmonic, by Theorem \ref{thm:biharmonicequations}
we have
\[
B^{N}=0\,\,\,and\,\,\,B^{T}=0.
\]
Let $\bar{f}=\left\langle X,V\right\rangle |S^{n+1}$, where $X$
is the position vector of $R^{n+2}$ and $V$ is a fixed vector in
$R^{n+2}$. We have
\[
\bar{\nabla}\bar{\nabla}\bar{f}=-\bar{f}\bar{g},
\]
and
\[
\bar{\nabla}\bar{\nabla}\bar{\nabla}\bar{f}=-d\bar{f}\otimes\bar{g},
\]
where $\bar{g}=g_{S^{n+1}}$. Let $f=\left\langle X,V\right\rangle |M$.
By Theorem \ref{thm:bilaplacianrestriction} we have
\[
\Delta\Delta f=\Delta\left(tr_{M}\bar{\nabla}\bar{\nabla}\bar{f}\right)-nH\left\langle \bar{\nabla}\bar{f},N\right\rangle +H^{2}f.
\]
Integrating and using the divergence theorem we obtain
\[
0=0-\int_{M}nH\left\langle \bar{\nabla}\bar{f},N\right\rangle +\int_{M}H^{2}f.
\]
By Lemma \ref{lem:hessianrestriction} we have
\[
\Delta f=-nf+\left\langle \bar{\nabla}\bar{f},N\right\rangle H.
\]
Multiplying by $n$, integrating and using the divergence theorem
we obtain
\[
0=-\int_{M}n^{2}f+\int_{M}nH\left\langle \bar{\nabla}\bar{f},N\right\rangle .
\]
By the above equations we have
\[
\int_{M}\left(n^{2}-H^{2}\right)f=0.
\]
Since $M$ lies in a closed hemisphere of $S^{n+1}$ we can find $V$
in $R^{n+2}$ such that $f\geq0$. Since $n^{2}-H^{2}$ does not change
sign we have
\[
\left(n^{2}-H^{2}\right)f=0.
\]
First suppose $H^{2}=n^{2}$ everywhere. In this case by Theorem 2.10
in \cite{BMO2008} we conclude that $M$ is the small hypersphere
$S^{n}\left(1/\sqrt{2}\right)$. Now suppose $H^{2}\neq n^{2}$ at
some point of $M$. In this case we have $f=0$ on an open subset
of $M$, that is, an open subset of $M$ lies in the equator of the
hemisphere. Since this subset is minimal, by Theorem 1.3 in \cite{BO2019}
we find that the whole $M$ is minimal. We conclude that $M$ is the
equator of the hemisphere.
\end{proof}

\section{Further applications}

Let $M^{n}$ be a hypersurface in $S^{n+1}$ and let 
\[
x_{i}=\left\langle X,E_{i}\right\rangle |M^{n},
\]
where $X$ is the position vector of $R^{n+2}$ and $\left\{ E_{i}\right\} _{i=1}^{n+2}$
is the standard basis of $R^{n+2}$.

Takahashi \cite{T1966} proved that $M^{n}$ is minimal if and only
if $\Delta x_{i}=-nx_{i}$ for $1\leq i\leq n+2$. As an application
of Theorem \ref{thm:bilaplacianrestriction} we obtain a similar result
for biharmonic hypersurfaces.
\begin{thm}
\label{thm:biharmonicsphere}Let $M^{n}$ be a hypersurface in $S^{n+1}$.
Then $M^{n}$ is biharmonic if and only if
\[
\Delta\Delta x_{i}=\left(n^{2}+H^{2}\right)x_{i}-2nH\left\langle N,E_{i}\right\rangle \,\,\,for\,\,\,1\leq i\leq n+2.
\]
\end{thm}
\begin{proof}
Let $\bar{f}=\left\langle X,E_{i}\right\rangle |S^{n+1}$. We have
\[
\bar{\nabla}\bar{\nabla}\bar{f}=-\bar{f}\bar{g},
\]
and
\[
\bar{\nabla}\bar{\nabla}\bar{\nabla}\bar{f}=-d\bar{f}\otimes\bar{g},
\]
where $\bar{g}=g_{S^{n+1}}$. Let $f=\left\langle X,E_{i}\right\rangle |M$.
By Theorem \ref{thm:bilaplacianrestriction} and Lemma \ref{lem:hessianrestriction}
we have
\[
\Delta\Delta f=-n\Delta f-nH\left\langle \bar{\nabla}\bar{f},N\right\rangle +H^{2}f+\left\langle B^{N}N-B^{T},\bar{\nabla}\bar{f}\right\rangle ,
\]
and
\[
\Delta f=-nf+\left\langle \bar{\nabla}\bar{f},N\right\rangle H.
\]
We find that
\[
\Delta\Delta f=\left(n^{2}+H^{2}\right)f-2nH\left\langle \bar{\nabla}\bar{f},N\right\rangle +\left\langle B^{N}N-B^{T},\bar{\nabla}\bar{f}\right\rangle .
\]
Since $\bar{\nabla}\bar{f}=proj_{TS^{n+1}}E_{i}$ we have
\[
\Delta\Delta x_{i}=\left(n^{2}+H^{2}\right)x_{i}-2nH\left\langle N,E_{i}\right\rangle +\left\langle B^{N}N-B^{T},E_{i}\right\rangle \,\,\,for\,\,\,1\leq i\leq n+2.
\]
The result follows from Theorem \ref{thm:biharmonicequations}.
\end{proof}
This result was obtained in a completely different way by Caddeo-Montaldo-Oniciuc
(Proposition 4.1 in \cite{CMO2002}). The statement is different but
equivalent.

As an application of Theorem \ref{thm:biharmonicsphere} we obtain
a sufficient condition for a biharmonic hypersurface in $S^{n+1}$
to be minimal.
\begin{cor}
Let $M^{n}$ be a biharmonic hypersurface in $S^{n+1}$. If there
exists a function $\phi$ on $M^{n}$ such that
\[
\Delta\Delta x_{i}=\phi x_{i}\,\,\,for\,\,\,1\leq i\leq n+2,
\]
then $M^{n}$ is minimal.
\end{cor}
\begin{proof}
By Theorem \ref{thm:biharmonicsphere} we have
\[
\phi\left\langle X,E_{i}\right\rangle =\left(n^{2}+H^{2}\right)\left\langle X,E_{i}\right\rangle -2nH\left\langle N,E_{i}\right\rangle =\,\,\,for\,\,\,1\leq i\leq n+2.
\]
We find that
\[
2nHN=\left(n^{2}+H^{2}-\phi\right)X.
\]
The result follows since $N$ is tangent to $S^{n+1}$ and $X$ is
normal to $S^{n+1}$.
\end{proof}

\lyxaddress{Departamento de Matemática, Universidade Federal do Espírito Santo,
Vitória, Brazil. E-mail: matheus.vieira@ufes.br}
\end{document}